\documentclass{amsart}

\usepackage{amsmath,amssymb}
\usepackage{bm}
\usepackage{graphicx}
\usepackage{ascmac}
\usepackage{amsthm}
\usepackage{color} 

\theoremstyle{plain}
\newtheorem{thm}{Theorem}[section]

\newtheorem{prop}[thm]{Proposition}
\newtheorem{lmm}[thm]{Lemma}

\theoremstyle{definition}

\DeclareMathOperator{\an}{an}
\DeclareMathOperator{\hyb}{hyb}
\DeclareMathOperator{\Pic}{Pic}
\DeclareMathOperator{\NA}{NA}
\DeclareMathOperator{\MA}{MA}
\DeclareMathOperator{\GL}{GL}
\DeclareMathOperator{\Frac}{Frac}
\DeclareMathOperator{\Spec}{Spec}
\DeclareMathOperator{\diag}{diag}
\DeclareMathOperator{\ord}{ord}
\DeclareMathOperator{\NS}{NS}

\newcommand{\bA}{\mathbb{A}}
\newcommand{\bC}{\mathbb{C}}
\newcommand{\bD}{\mathbb{D}}
\newcommand{\bG}{\mathbb{G}}
\newcommand{\bN}{\mathbb{N}}
\newcommand{\bP}{\mathbb{P}}
\newcommand{\bQ}{\mathbb{Q}}
\newcommand{\bR}{\mathbb{R}}
\newcommand{\bZ}{\mathbb{Z}}

\newcommand{\cC}{\mathcal{C}}
\newcommand{\cD}{\mathcal{D}}

\newcommand{\cF}{\mathcal{F}}
\newcommand{\cL}{\mathcal{L}}

\newcommand{\cO}{\mathcal{O}}
\newcommand{\cU}{\mathcal{U}}
\newcommand{\cX}{\mathcal{X}}

\newcommand{\fA}{\mathfrak{A}}

\usepackage[truedimen,top=25truemm,bottom=25truemm,left=25truemm,right=25truemm]{geometry}

\title{Hybrid dynamics of hyperbolic automorphisms of K3 surfaces}
\author{Reimi Irokawa}
\address{NTT Institute for Fundamental Mathematics, NTT Communication Science Laboratories, Nippon Telegraph and Telephone Corporation, 3-9-11 Midori-cho, Musashino-shi, Tokyo 180-8686, Japan}
\email{reimi.irokawa@ntt.com}

\date{today}

\begin{document}
\begin{abstract}
    We study degenerating families of hyperbolic dynamics over complex K3 surfaces by means of the theory of hybrid spaces by Boucksom, Favre, and Jonsson. For an analytic family of hyperbolic automorphisms $\{f_t: X_t\to X_t\}_{t\in\bD^*}$ over K3 surfaces $X_t$ that is possibly meromorphically degenerating at the origin, we consider the family of invariant measures $\{\eta_t\}$ on $X_t$ constructed by Cantat. The family $f_t$ induces a hyperbolic automorphism $f_{\bC((t))}^{\an}:X_{\bC((t))}^{\an}\to X_{\bC((t))}^{\an}$ over the induced non-archimedean K3 surface, where we also have a measure $\eta_0$ by Filip. Our main theorem states the weak convergence of $\{\eta_t\}$ to $\eta_0$ as $t\to0$ over the induced so-called hybrid space.
\end{abstract}
\maketitle

\section{Introduction}\label{section: hybridK3 - Introduction}

In this paper, we study degenerating families of hyperbolic dynamics over complex K3 surfaces by means of the theory of hybrid spaces. We briefly state the setup and main result here, and details will be given in Section \ref{section: hybrid K3 -general setup}.

Let $X$ be a complex projective K3 surface and $f: X\to X$ is a hyperbolic automorphism, i.e., the specral radius of the induced isometry $f^*:H^{1,1}(X,\bR)\to H^{1,1}(X,\bR)$ is strictly greater than $1$. Dynamical systems of such automorphisms are studied by Cantat \cite{Cantat:2001} since this is the only possible case where automorphisms of K3 surfaces have positive topological entropy. In particular, he constructed a positive measure on $X$ invariant under the action of $f$ by means of pluripotential theory over $X$, which is the unique measure of maximal entropy.

The main purpose of this paper is to understand the limiting behavior of a degenerating family of such dynamical systems. Let $\bD$ be the unit disk of $\bC$, and $\bD^*=\bD\setminus\{0\}$ be the punctured unit disk, which we regard as a parameter space. We consider a family of K3 surfaces $X\to\bD^*$ with a family of hyperbolic automorphisms $f: X\to X$, i.e., $X_t$ is a complex K3 surface for each $t\in\bD^*$, a map $f$ preserves the fiber structure $X\to \bD^*$, and each $f_t=f|_{X_t}: X_t\to X_t$ is a hyperbolic automorphism. We assume the family $(X,f)$ meromorphically degenerates at the origin. Namely, if we denote the set of holomorphic functions around $t=0$ in $\bD$ by $\cO_{\bD,0}$, $X$ and $f$ can be regarded as a scheme and its automorphism over $\Spec\Frac(\cO_{\bD,0})$, respectively. In this setting, we obtain a family of measures $\{\eta_t\}_{t\in\bD^*}$ on $X$ such that $\eta_t$ is invariant under the action of $f_t$ for each $t\in\bD^*$. We study its weak limit as $t\to0$, but we work on the so-called {\it hybrid spaces} attached to the given degenerating family $X$.

The theory of hybrid spaces is invented by Boucksom, Favre, and Jonsson in a series of papers \cite{Sebastien:2015uz}, \cite{BOUCKSOM:2015fr}, and \cite{Sebastien:2017le} to study the degeneration of complex manifolds. Later, it is applied to studying complex dynamics by Favre \cite{Favre:2020rz}. Roughly speaking, this theory enables us to regard non-archimedean objects as degenerating limits of complex objects. The non-archimedean objects are naturally induced from the degenerating families: we regard $X$ as a variety over $\Frac(\cO_{\bD,0})$, which is naturally regarded as a subring of the Laurent series field $\bC((t))$. We introduce the $t$-adic valuation to $\bC((t))$ and take a base extensions $X_{\bC((t))}$ and $f_{\bC((t))}:X_{\bC((t))}\to X_{\bC((t))}$. The required non-archimedean object is its Berkovich analytification $f_{\bC((t))}^{\an}:X_{\bC((t))}^{\an}\to X_{\bC((t))}^{\an}$. 

The dynamical systems of hyperbolic automorphisms over non-archimedean projective K3 surfaces are studied by Filip \cite{Filip:2019ij}. In particular, for the above automorphism $f_{\bC((t))}^{\an}: X_{\bC((t))}^{\an}$, he constructed a positive measure $\eta$ invariant under the action of $f^{\an}$. 

From the family $X$ of K3 surfaces, we can associate a so-called hybrid K3 surface $X^{\hyb}$ attached to $X$ as in \cite{Favre:2020rz}. For a fixed real number $0<r<1$, the space $X^{\hyb}$ admits an surjection $X^{\hyb}\to\bar{\bD}_r$, where $\bar{\bD}_r$ denotes the complex closed disc of radius $r$ with centered at $0$ such that
\begin{align*}
\psi_t: X^{\hyb}_t\simeq \begin{cases}
X_t &\text{ for }t\neq0,\\
X^{\an}_{\bC((t))} &\text{ for }t=0.
\end{cases}
\end{align*}
For each $t\neq0$, the invariant measure for the complex dynamics $(X_t,f_t)$ by Cantat is denoted by $\eta_t$ and that for the non-archimedean dynamics $(X_{\bC((t))}^{\an},f^{\an}_{\bC((t))})$ by Filip by $\eta_0$, so that we have a family of measures $\{(\psi_t)_*\eta_t\}_{t\in\bar{\bD}_r}$ on $X^{\hyb}$. We assume that, by taking positive scaling, all of $\eta_t$'s are probability measures. Then, the main theorem is as follows.
\begin{thm}[Theorem \ref{thm: hybridK3 - main theorem - main theorem}]\label{thm: hybridK3-weakconvergence-invariantmeasure}
    In the above notation, we have the following weak convergence of measures:
    \begin{align*}
        \lim_{t\to0}\psi_{t,*}\eta_t=\psi_{0,*}\eta_0.
    \end{align*}
\end{thm}

As written above, the theory of hybrid spaces is first applied to the study of degeneration of complex dynamics by Favre \cite{Favre:2020rz}. He shows a similar result of weak convergence of canonical measures for a degenerating family of holomorphic endomorphisms over projective spaces $\bP^n(\bC)$. The author shows a similar result for degenerating families of complex H{\'e}non mappings \cite{irokawa2023hybrid}. As an application of Favre's result, Bianchi and Okuyama \cite{Bianchi2018DegenerationOQ} study degeneration of quadratic polynomial endomorphisms over $\bP^2(\bC)$ when they degenerate to a H{\'e}non map. Also, we should note that before the work of Favre, DeMarco and Faber \cite{DE-MARCO:2014tn} studies the degeneration of rational functions on $\bP^1(\bC)$ by a similar technique, where they used a Berkovich projective line over the completion of the Puiseux series field $\bC\{\{t\}\}$, and Okuyama \cite{okuyama2024degenerating} fixed their proof.

Our result treats a family of hyperbolic automorphisms over K3 surfaces, where both morphisms and phase spaces move along $t$. Also, all the known results treat degenerating morphisms on either $\bA^{2}(\bC)$ or $\bP^{n}(\bC)$, both of which have global coordinates while K3 surfaces do not. These differences make our argument more delicate, as seen in Section \ref{section: hybridK3 - main theorem}. Indeed, we need two different snc models of $X$ for the domain, and the range of $f$ to make $f$ extend as a morphism to the snc models.

\subsection*{Organization of this paper}
We note that all the proofs are presented in Section \ref{section: hybridK3 - Weak convergence of mixed Monge-Ampere measures of model metrics} and \ref{section: hybridK3 - main theorem}. 

In Section \ref{section: hybrid K3 -general setup}, we present background knowledge for state our main result precisely, including dynamics over complex and non-archimedean K3 surfaces and the definition of hybrid spaces. In Section \ref{section: hybridK3 - backgrounds}, we collect facts from the theory of hybrid spaces, most of which are non-dynamical. By using them, we show a key result to show the main theorem in Section \ref{section: hybridK3 - Weak convergence of mixed Monge-Ampere measures of model metrics}, the weak convergence of mixed Monge-Amp{\`e}re measures. This is a non-dynamical result either, and valid for more general varieties. We conclude this paper in Section \ref{section: hybridK3 - main theorem} by showing the main result.

\subsection*{Acknowledgement}
I would like to thank Professor Mattias Jonsson for his kind advice and discussion. Especially, he pointed out my mistake on the proof of Proposition \ref{prop: hybridK3-weakconvergence-invariantmeasure-uniformestimation} and help me fix it.

\section{General setup and the dynamics over complex and non-archimedean K3 surfaces} \label{section: hybrid K3 -general setup}
In this section, we collect known results for stating our main result precisely, which is done in Section \ref{section: hybridK3 - main theorem}. More precisely, the organization of this section is as follows: in Section \ref{subsection: backgrounds - dynamics of K3 surfaces}, we recall the basics of dynamics over complex and non-archimedean K3 surfaces. In Section \ref{subsection: hybridK3 - backgrounds - Berkovich geomety}, we recall a general theory of Berkovich analytification, which includes the theory of hybrid spaces. In section \ref{subsection: background - hybrid K3 dynamics}, we apply Section \ref{subsection: hybridK3 - backgrounds - Berkovich geomety} to our specific situation.

\subsection{Complex and non-archimedean K3 surfaces} \label{subsection: backgrounds - dynamics of K3 surfaces}
We recall basic results from algebraic and non-archimedean geometry and dynamics over them first. Algebro-geometric results for K3 surfaces are well-known, for instance, from \cite{Huybrechts_2016}. For results on dynamical systems over complex K3 surfaces, we refer \cite{Cantat:2001}. For those over non-archimedean ones, we refer \cite{Filip:2019ij}.

A complex {\it K3 surface} is a compact complex manifold of dimension $2$ whose canonical bundle is trivial and irregularity is zero. For a K3 surface $X$, $H^2(X,\bZ)$ is a finitely generated free $\bZ$-module. Hence, the N{\'e}ron-Severi group $\NS(X)$ is also a finitely generated free $\bZ$-module, equipped with a natural bilinear form defined from the cup product. Let $\Pic(X)$ be the Picard group of $X$. Then, since the irregularity is zero, there exists an isomorphism between the N{\'e}ron-Severi group and the Picard group of $X$. Especially, $\Pic(X)$ is torsion-free. We note that the bilinear form on $\NS(X)$ is also compatible with that on $\Pic(X)$ defined by the intersection product. In this paper, we only treat projective K3 surfaces. Then, the bilinear form on $\NS(X)$ is non-degenerate and of sign $(1,r-1)$ where $r$ is a rank of $\NS(X)$.

Since the action $f^*:\Pic(X)\otimes_{\bZ}\bR\to\Pic(X)\otimes_{\bZ}\bR$ preserves the lattice structure of $\Pic(X)$ and the bilinear form, we have a classification of the automorphism $f$. Namely, $f$ must satisfy exactly one of the following properties:
\begin{itemize}
    \item $f$ is {\it periodic} (or {\it elliptic}), i.e., $f^*$ preserves the positive cone ($\{x\in \Pic(X)\otimes_{\bZ}\bR; x^2>0\}$) and $f^*$ is a rotation, in which case we must have $(f^*)^n=f^*$ for some natural number $n>1$;
    \item $f$ is {\it parabolic}, i.e.,  $f^*$ admits a unique eigenvector $v$ up to scalar such that $v^2=0$ and its eigenvalue is exactly $1$, and all the other eigenvalues are of absolute values $1$;
    \item $f$ is {\it hyperbolic}, i.e., $f^*$ admits a unique eigenvector $v$ up to scalar such that $v^2=0$ and its eigenvalue has absolute value greater than $1$.
\end{itemize}
We remark that, according to the above classification, $f$ has an eigenvalue whose absolute value is greater than $1$ if and only if $f$ is hyperbolic. That is, once we find an eigenvalue $\lambda>1$, it must be the only eigenvalue whose absolute value is greater than $1$. In this case, the corresponding eigenspace must be $1$-dimensional, $\lambda^{-1}$ is also an eigenvalue with $1$-dimensional eigenspace, and all the other eigenvalues are of absolute value $1$.

In relation to dynamics, the following Gromov-Yomdin's theorem \cite{gromov2003entropy} is important: for an automorphism $f: X\to X$, its topological entropy is the logarithm of the spectral radius of the induced isometry $f^*:H^{1,1}(X,\bR)\to H^{1,1}(X,\bR)$. In this case, $H^{1,1}(X,\bR)\simeq\Pic(X)\otimes_{\bZ}\bR$. Therefore, there is only one possible case in the above classification where the topological value is strictly positive: when $f$ is hyperbolic. For such a morphism, Cantat \cite{Cantat:2001} shows the following fact:
\begin{prop}[\cite{Cantat:2001}, Th{\'e}or{\`e}me 2.4]\label{proposition: hybridK3 - canocurrents for cpx K3 surfaces}
    Let $X$ be a complex projective K3 surface and $f: X\to X$ be a hyperbolic automorphism. Then, there exist two positive closed (1,1)-currents $\eta_{\pm}$ on $X$ such that
\begin{align*}
    f_*\eta_{\pm}=\lambda^{\pm 1}\eta_{\pm}.
\end{align*}
These currents are unique up to positive scaling, respectively, and they both locally have continuous potential functions.
\end{prop}
Since they both have continuous potentials, their wedge product $\eta=\eta_+\wedge\eta_-$ is a well-defined and positive measure on $X$. By definition of $\eta_{\pm}$, the measure $\eta$ is invariant under the action of $f$. We note that the homology class of $\eta_{\pm}$ must be in the eigenspace of the eigenvalue $\lambda^{\pm1}$ respectively, and their self-intersection products are $0$. We hence deduce that $\eta_{\pm}\wedge \eta_{\pm}=0$.

When we restrict surfaces to algebraic ones, we can generalize the notion of K3 surfaces for any base field $K$. For convenience, we assume that the characteristic of $K$ is $0$. a {\it K3 surface} over $K$ is a connected complete non-singular algebraic variety of dimension $2$ whose canonical bundle is trivial and irregularity is zero. Note that when $K=\bC$, the complex analytification of K3 surfaces in this sense are the complex projective K3 surfaces discussed above.

As in the complex case, $\NS(X)\simeq\Pic(X)$ has a natural bilinear form with sign $(1,r-1)$. Hence the action $f^*:\Pic(X)\otimes_{\bZ}\bR\to\Pic(X)\otimes_{\bZ}\bR$ has the same classification as above.

Now we take $K$ as a field of characteristic $0$ equipped with complete, non-trivial, and non-archimedean valuation. Let $X$ be a K3 surface and $f: X\to X$ be a hyperbolic automorphism. For an algebraic variety and its morphism, we can consider the Berkovich analytification of them. We postpone its definition for a while, and so far, we consider them as some topological space and its homeomorphism, denoted by $X^{\an}$ and $f^{\an}: X^{\an}\to X^{\an}$. Filip \cite{Filip:2019ij} studied the dynamics of such a pair $(X^{\an},f^{\an})$. Especially he constructed the currents and the measure similar to the ones above by Cantat:
\begin{prop}\label{proposition: hybridK3 - canocurrents for NA K3 surfaces}
    Let $X$ be a K3 surface over a field $K$ of characteristic $0$ equipped with a non-archimedean and non-trivial valuation and $f:X\to X$ be its hyperbolic automorphism. Then, there exist two positive closed (1,1)-currents $\eta_{\pm}$ on $X^{\an}$ such that
    \begin{align*}
        f^*\eta_{\pm}=\lambda^{\pm1}\eta_{\pm}.
    \end{align*}
    They both locally have continuous potential functions.
\end{prop}
Here, the current is in the sense of Chambert-Loir and Ducros; they \cite{chambertloir2012formes} constructed a theory of forms and currents over Berkovich varieties. They also define the wedge product of currents with continuous potentials. We apply them to $\eta_{\pm}$ to obtain a measure $\eta=\eta_+\wedge\eta_-$ on $X^{\an}$, which is invariant under the action of $f$ as in the complex case. The equality $\eta_{\pm}\wedge\eta_{\pm}=0$ in the complex case comes from the cohomological argument, so it holds in the non-archimedean case, too.

\subsection{Non-archimedean and hybrid analytification of $X$} \label{subsection: hybridK3 - backgrounds - Berkovich geomety}
We consider two types of Berkovich analytification for the variety $X$ as above.

The first is analytification as a variety over $\bC((t))$. We note that the coefficients of the defining equations of $X$ in $\bP^{N}$ are naturally regarded as those in $\bC((t))$. We fix $0<r<1$, take a base of the (multiplicative) $t$-adic valuation as the given $r$, and denote it by $|\cdot|_r$, i.e., 
\begin{align*}
    \left|\sum_{n=m}^{\infty}a_nt^n\right|_r:=r^{m},
\end{align*}
where $a_m\neq0$. Then, $X$ can be regarded as a projective variety over the non-archimedean field $\bC((t))$, which is denoted by $X_{\bC((t))}$. For this $X_{\bC((t))}$, we define the Berkovich analytification $X_{\bC((t))}^{\an}$ as follows: the underlying set consists of the multiplicative seminorms on all the residue fields of $X_{\bC((t))}$ whose restriction to $\bC((t))$ is the given $t$-adic valuation. That is, if we denote the residue field at $x$ by $\kappa(x)$, an element of $X^{\an}_{\bC((t))}$ is written as a pair $(x,|\cdot|)$ of a point $x\in X_{\bC((t))}$ and a norm $|\cdot|:\kappa(x)\to\bR$ that is multiplicative and $|a|=|a|_r$ for any $a\in \bC((t))(\subset \kappa(x))$. The topology is a coarsest one such that, for any affine open subset $U\subset X_{\bC((t))}$ and any regular function $f$ on $U$, the evaluation map $(x,|\cdot|)\mapsto |\overline{f}_x|$ is continuous, where $\overline{f}_x$ is the class of the germ of $f$ at $x$ in $\kappa(x)$. We denote this value by $|f(x)|$. Since we assume that $X_{\bC((t))}$ is a proper variety over $\bC((t))$, the analytification $X_{\bC((t))}^{\an}$ is compact and Hausdorff.

We remark that the following several kinds of seminorms are important in our later discussion. First, the set of closed point of $X_{\bC((t))}$ is naturally embedded to $X_{\bC((t))}^{\an}$. Indeed, $\kappa(x)$ is a finite extension of $\bC((t))$ when $x$ is a closed point, so there is a unique norm extending the one on $\bC((t))$. We call points corresponding to closed points of $X_{\bC((t))}$ {\it classical points}.

The second one is a so-called {\it divisorial point}. Divisorial points are norms on the field of rational functions $K(X_{\bC((t))})$, i.e., the residue field of the generic point of $X_{\bC((t))}$. Take any snc model $\cX$ of $X$, i.e., $\cX$ is a complex manifold of dimension $n+1$, there is a proper map $\pi_{\cX}:\cX\to\bD$ with an isomorphism $\pi_{\cX}^{-1}(\bD^*)\simeq X$, and $\cX_0:=\pi^{-1}\{0\}$ defines a simple normal crossing divisor of $\cX$. We write $\cX_0=\sum_E b_E E$, where $E$ runs over all the irreducible components of $\cX_0$. Then, the local ring at the generic point of $E$ is a discrete valuation ring, and thus, it defines a discrete valuation. Its field of fraction is $K(X_{\bC((t))})$ and we have $\ord_E(t)=b_E$, where we denote the corresponding additive valuation by $\ord_E$. Finally, if we set for any $f\in K(X_{\bC((t))})$,
\begin{align*}
    |f(x_E)|:=r^{\ord_E(f)/b_E},
\end{align*}
this defines a point in $X_{\bC((t))}^{\an}$. We call such $x_E$ a divisorial point.

The second analytification is the hybrid one. In order to define it, we prepare some notation. The {\it hybrid norm} $|\cdot|_{\hyb}$ on $\bC$ is defined as
\begin{align*}
    |c|_{\hyb}=\max(1,|c|),
\end{align*}
where $c$ is any non-zero complex number, and $|\cdot|$ in the right hand side is the euclidean norm on $\bC$. This $|\cdot|_{\hyb}$ defines a sub-multiplicative norm on $\bC$. We define a subring $A_r$ of $\bC((t))$ consisting of all the element $f=\sum_{n=m}^{\infty}a_nt^n\in\bC((t))$ such that $\sum|a_n|_{\hyb}r^n<\infty$. If we introduce a sub-multiplicative norm $\|\cdot\|_{\hyb,r}$ on $A_r$ by
\begin{align*}
    \left\|\sum_{n=m}^{\infty}a_nt^n\right\|:=\sum_{n=m}^{\infty}|a_n|_{\hyb}r^n,
\end{align*}
then $A_r$ becomes a Banach ring. We take its Berkovich spectrum $\cC_{\hyb}(r)$, the set of all the bounded multiplicative seminorms on $A_r$ with the topology of the pointwise convergence.

For a point $t_0\in\overline{\bD}^*_r$, define a seminorm $\tau(t_0)$ by
\begin{align*}
    \sum_{n=m}^{\infty}a_nt^n\mapsto \left|\sum_{n=m}^{\infty}a_nt_0^n\right|^{\frac{\log r}{\log |t_0|}},
\end{align*}
where $|\cdot|$ is the euclidean norm on $\bC$. For $t_0=0$, we define $\tau(0)$ to be the $t$-adic norm $|\cdot|_r$. By Proposition 1.1 of \cite{Favre:2020rz}, $\tau(t_0)$ for any $t_0\in\overline{\bD}_r$ defines a point in $\cC_{\hyb}(r)$ and the map $\tau:\overline{\bD}_r\to\cC_{\hyb}(r)$ is actually homeomorphic.

As coefficients of the defining equation of $X$ are defined on the unit punctured disc, they are naturally regarded as elements of $A_r$. We set $X_{A_r}$ as the projective variety over $A_r$ defined in this way.

For any affine variety $V=\Spec B$ where $B$ is an $A_r$-algebra of finite type, we define its hybrid analytification $V^{\hyb}$ in the following way; the underlying set consists of the multiplicative seminorm on $B$ whose restriction to $A_r$ is bounded by the valuation $\|\cdot\|_{\hyb,r}$. For any $x\in V$ and $f\in B$, we denote the value of $f$ evaluated at $x$ by $|f(x)|$. The topology on $V^{\hyb}$ is the weakest one such that $x\mapsto|f(x)|$ is continuous for any $f\in B$. Then, for $X_{A_r}$, its hybrid analytification, denoted by $X^{\hyb}$, is now defined by patching the hybrid analytification of its affine open cover. 

The structural morphism $X_{A_r}\to\Spec A_r$ induces a continuous map $\pi_{\hyb}:X^{\hyb}\to\cC_{\hyb}(r)$. Note that, on an affine open subset $V=\Spec B$ of $X_{A_r}$, the structural morphism is $A_r\to B$, and the map $V^{\hyb}\to\cC_{\hyb}(r)$ is a restriction map; since a point $x\in U^{\hyb}$ is a multiplicative seminorm, we can restrict it to $A_r$. Then, since we require the restriction of $x$ to $A_r$ to be bounded by the norm on $A_r$, this restriction also defines a point on $\cC_{\hyb}(r)$. 
By Theorem 1.2 of \cite{Favre:2020rz}, We have a natural homeomorphism $\psi_0:\pi_{\hyb}^{-1}\{\tau(0)\}\to X_{\bC((t))}^{\an}$ and $\psi:\pi_{\hyb}^{-1}\left(\tau(\bar{\bD}_r^*)\right)\to\bar{X}_r$ so that $\pi_{\hyb}\circ\psi=\tau\circ\pi$. Especially, when we put $\psi_t:=\psi|_{\pi^{-1}\{\tau(t)\}}$ for $t\neq0$, we have a family of homeomorphism $\{\psi_t\}_{t\in\bar{\bD}_r}$
\begin{align*}
    \psi_t:\pi_{\hyb}^{-1}\{\tau(t)\}\to \begin{cases}
        X_t &\text{ for }t\neq 0, \\
        X_{\bC((t))}^{\an} &\text{ for }t=0.
    \end{cases}
\end{align*}

For $t_0\neq0$, a point $z_0\in X_{t_0}$ corresponds to a multiplicative seminorm
\begin{align*}
    \phi\mapsto\left|\phi(t_0,z_0)\right|^{\frac{\log r}{\log|t_0|}}
\end{align*}
for any rational function $\phi$ on $X_{A_r}$ via $\psi_{t_0}$.

Define a function $\eta:X^{\hyb}\to[0,1]$ for later discussion. For $x\in X$, put
\begin{align*}
    \eta(\psi(x))=\frac{\log r}{\log|\pi(x)|},
\end{align*}
which appears in the definition of the corresponding multiplicative seminorm $\psi(x)$ and takes values in $(0,1]$. For $x\in \pi_{\hyb}^{-1}\{\tau(0)\}$, we put $\psi(x)=0$. The functon $\eta$ is continuous on $X^{\hyb}$.

We remark that analytic varieties $X^{\an}_{\bC((T))}$ and $X^{\hyb}$ are not just topological spaces but also locally ringed spaces, i.e., they have sheaves of analytic functions over the underlying spaces. Especially on each open set of $X_{\bC((t))}^{\an}$, we can attach the set of analytic functions, and it naturally contains regular functions in the sense of algebraic geometry when the open set is affine. Every algebraic morphism is not only continuous but also analytic. 

\subsection{Hybrid K3 surfaces and its dynamics}\label{subsection: background - hybrid K3 dynamics}
Let $X\to \bD^*$ be an analytic family of projective K3 surfaces possibly meromorphically degenerating at the origin. Set $K=\Frac{\cO_{\bD,0}}$ be the field of fractions of analytic functions defined around $0$. Take a base extension of $X$ to $K$ to obtain a K3 surface $X_K$ over $K$. We call an analytic family of automorphisms $\{f_t:X_t\to X_t\}_{t\in\bD^*}$ is {\it hyperbolic} if $f_K:X_K\to X_K$ is hyperbolic.

Note that in this case, $f_t$ is hyperbolic for any $t$ small enough. To see this, we take a basis $L_{1,K},\ldots,L_{k,K}$ of $\Pic(X_K)$. We assume all of the $L_{i, K}$'s are ample for the later discussion. By shrinking a disc if necessary, there exist ample line bundles $L_1,\ldots, L_k$ of $X$ such that their base extension to $K$ are $L_{1, K},\ldots, L_{k, K}$. We denote by $L_{i,t}$ the line bundle given by the restriction of $L_i$ to the K3 surface $X_t$. Then, $\{L_{i,t}\}_{i=1}^k$ forms a subspace of $\Pic(X)$. Since $f$ is hyperbolic, there exists a real number $\lambda>1$ and $\alpha_1,\ldots,\alpha_k$ such that $\alpha_1 L_{1,K}+\cdots+\alpha_k L_{k,K}$ is an eigenvector of $f^*_K$ with eigenvalue $\lambda$. Then, its specialization $\alpha_1 L_{1,t}+\cdots+\alpha_k L_{k,t}$ is also an eigenvector of $f^*_t$ with eigenvalue $\lambda$. By classification of automorphisms over $X_t$, $f$ must be hyperbolic if there is an eigenvalue of $f^*$ whose absolute value is greater than $1$. 

Fix $0<r<1$ so small that $f_t$ is hyperbolic for any $t\in\bar{\bD}_r$. Then, as we see in the previous subsection, we can take both its non-archimedean and hybrid analytification. Especially for the hybrid analytification, there exists a proper surjection $\pi_{\hyb}X^{\hyb}\to\cC_{\hyb}(r)$, a homeomorphism $\tau:\cC\to\bar{\bD}_r$, and a family of homeomorphisms $\{\psi_t\}$ with
\begin{align*}
    \psi_t:\pi_{\hyb}^{-1}\{\tau(t)\}\to \begin{cases}
        X_t &\text{ for }t\neq 0, \\
        X_{\bC((t))}^{\an} &\text{ for }t=0.
    \end{cases}
\end{align*}

We apply Proposition \ref{proposition: hybridK3 - canocurrents for cpx K3 surfaces} and Proposition \ref{proposition: hybridK3 - canocurrents for NA K3 surfaces} in our situation. For $t\in\bar{\bD}_r^*$, the hyperbolic automorphism $f_t:X_t\to X_t$ defines the positive closed $(1,1)$-currents $\eta_{\pm,t}$ on $X_t$ such that $f_t^*\eta_{\pm,t}=\lambda^{\pm1}\eta_{\pm,t}$. We remark that in Proposition \ref{proposition: hybridK3 - canocurrents for cpx K3 surfaces}, the currents satisfies equalities for the action $f_*$, but there is no big difference since $f_*=(f^{-1})^*$ as $f$ is an automorphism. For the induced non-archimedean dynamics $f^{\an}_{\bC((t))}:X_{\bC((t))}^{\an}\to X_{\bC((t))}^{\an}$, we also have the positive closed $(1,1)$-currents $\eta_{\pm,0}$ on $X_{\bC((t))}^{\an}$ such that $(f^{\an}_{\bC((t))})^*\eta_{\pm,0}=\lambda^{\pm1}\eta_{\pm,0}$. For any $t\in\bar{bD}_r$, we put $\eta_t:=\eta_{+,t}\wedge\eta_{-,t}$, which is a positive measure on $X_t$ for $t\neq0$ and $X_{\bD((t))}^{\an}$ for $t=0$. We assume that $\eta_t$'s are probability measures by taking a positive scaling if necessary. The main theorem, Theorem \ref{thm: hybridK3 - main theorem - main theorem} states the weak convergence of measures,
\begin{align*}
    \lim_{t\to0}\psi_{t,*}\eta_t=\psi_{0,*}\eta_0.
\end{align*}

Our strategy of the proof is as follows: first, we will present $\eta_t$ as weak limits of Monge-Amp{\`e}re measures of some metric over line bundles with nice properties, details of which are explained in Section \ref{subsection : backgrounds - model functions}. We note that this is the original construction of $\eta_t$ by Cantat \cite{Cantat:2001} in the complex case and Filip \cite{Filip:2019ij} in the non-archimedean case. If we denote the sequences of the Monge-Amp{\`e}re measures $\{\eta_{n,t}\}_n$ converging to $\eta_t$, we first show the weak convergence of the family $\{\psi_{t,*}\eta_{n,t}\}_t$ for each $n$, and then we show the theorem using it. 

We note that when we show the weak convergence of $\{\eta_{n,t}\}_t$ for each $n$, it is enough to show that there exists a family $\cF\subset C^0(X,\bR)$ that is dense with respect to the topology of uniform convergence such that 
\begin{align*}
    \lim_{t\to0}\int_{X^{\hyb}}\Phi\psi_{t,*}\eta_{n,t}=\int_{X^{\hyb}}\Phi\psi_{0,*}\eta_{n,0}
\end{align*}
holds for any $\Phi\in\cF$. For the dense subset $\cF$, we will use the so-called model functions in the sense of Favre \cite{Favre:2020rz}. We will explain it in the next section.

\section{Model functions and Monge-Amp{\`e}re measures}\label{section: hybridK3 - backgrounds}
In this section, we recall the notion of model functions and their properties, which appear in the proof of the main theorem. We mainly refer to \cite{Favre:2020rz}. We remark that we treat more general varieties than K3 surfaces, namely, any analytic family of projective varieties of arbitrary dimension, since results in Section \ref{section: hybridK3 - Weak convergence of mixed Monge-Ampere measures of model metrics} holds for them.

\subsection{Notations}
We prepare several notations. Let $\bD$ be the unit disc of $\bC$, and $\bD^*:=\bD\setminus\{0\}$ be the unit punctured disc. We denote the set of holomorphic functions defined around $0$ by $\cO_{\bD,0}$, that of holomorphic functions on $\bD$ by $\cO(\bD)$. If we denote the Laurent series field over $\bC$ by $\bC((t))$, both $\cO_{\bD, 0}$ and $\cO(\bD)$ can be regarded as subring of $\bC((t))$ naturally.

We consider an analytic family of smooth projective varieties $X$ over $\bD^*$. That is, $X$ is a closed subvariety of $\bP^N\times\bD^*$ such that the defining equation of $X$ have coefficients in $\cO(\bD)$, and for each $t\in\bD^*$, the variety $X_t$ is smooth, where $X_t=\pi^{-1}\{t\}$ with $\pi:X\to\bD^*$ the restriction of the second projection $\bP^N\times\bD^*\to\bD^*$. We fix an snc model $\cX$ with relatively very ample line bundle $\cL$ and its Fubini-Study metric $|\cdot|_*$. That is, there exists a closed immersion $\cX\to\bP^N\times\bD$ so that $\cL$ is the restriction of $\cO(1)\times\bD$, $\cX_0=\pi_{\cX}^{-1}\{0\}$ is a simple normal crossing divisor of $\cX$ where $\pi_{\cX}:\cX\to\bD$ is the restriction of the second projection $\bP^N\times\bD\to\bD$, and there is an isomorphism $\pi_{\cX}^{-1}(\bD^*)\simeq X$. 

We fix a real number $0<r<1$. Let $\overline{\bD}_r$ to be the closed disc centered at origin with radius $r$ and $\bar{\bD}_r^*:=\bar{\bD}_r\setminus\{0\}$. Define
\begin{align*}
    &\bar{X}_r:=\pi^{-1}(\bar{\bD}_r^*)=\bigcup_{0<|t|\leq r}X_t, \text{ and } \\
    &\bar{\cX}_r:=\pi_{\cX}^{-1}(\bar{\bD}_r)=\bigcup_{0<|t|\leq r}X_t\cup\cX_0.
\end{align*}

\subsection{Model functions} \label{subsection : backgrounds - model functions}
A {\it regular admissible datum} $\cF$ consists of
\begin{itemize}
    \item a positive integer $d\in\bN$;
    \item a vertical fractional ideal sheaf $\fA$ in $\cX$ such that $t^N\fA\subset \cO_{\cX}$ for large $N$;
    \item a finite set of meromorphic sections $\tau_0,\ldots\tau_l$ of $\cL^d$ defined in a neighborhood of $\bar{\cX}_r$ generating $\fA$;
    \item a log resolution $p:\cX'\to\cX$ of $\fA$.
\end{itemize}
For such a model $\cF$, we can attach the so-called {\it model function}. It is a continuous real-valued function on the hybrid space $X^{\hyb}$. We construct the model functions on both $\bar{X}_r$ and $X_{\bC((t))}^{\an}$ first and then define the model function on $X^{\hyb}$ by using the homeomorphism $\psi$.

we define the (archimedean) model function $\varphi_{\cF}:\bar{X}_r\to\bR_{\geq0}$ by
\begin{align*}
    \varphi_{\cF}(x):=\log\max\{|\tau_0|_*,\ldots,|\tau_l|_*\}.
\end{align*}

On the other hand, the family of varieties $X$ induces a projective variety $X_{\bC((t))}$ over $\bC((t))$ just by regarding the defining equations as elements of $\bC((t))$ via the Laurent expansion. By taking the Berkovich analytification, we have an analytic variety $X_{\bC((t))}^{\an}$. 

We define the model function $g_{\cF}$ attached to the model $\cF$ on the non-archimedean space $X_{\bC((t))}^{\an}$ using local generator of $\fA$.

Since $X_{A_r}$ is projective, there is a closed embedding $X_{A_r}\to\bP^N_{A_r}$. With a coordinate $[X_1:\cdots:X_N]$ of $\bP^N_{A_r}$, we define an affine open subset $U_i$ by the intersection of $X_r$ and $\{X_i\neq0\}$ in $\bP^N_{A_r}$. By Lemma 2.6 of \cite{Favre:2020rz}, there exist $f^{(i)},g_1^{(i)},\ldots,g_{l(i)}^{(i)}\in A_r$ such that $f^{(i)}\cdot\fA(U_i)$ is generated by $g_1^{(i)},\ldots,g_{l(i)}^{(i)}$. Then, for any $x\in U_i^{\hyb}$, define
\begin{align*}
    \log|\fA|(x)=\inf\{\log|g_j^{(i)}(x)|,j=1,\ldots,l\}-\log|f^{(i)}(x)|,
\end{align*}
where $|g_j^{(i)}(x)|$ and $|f^{(i)}(x)|$ denote the seminorms evaluated by $x$. This definition does not depend on the choice of generators or the embedding $X_{A_r}\to\bP^N_{A_r}$, so that $\log|\fA|$ defines a real-valued continuous function on $X^{\an}_{\bC((t))}$ by restricting it to $\pi^{-1}\{\tau(0)\}$. We denote this function by $g_{\cF}$. 

With the two model functions attached to $\cF$, we define the model function $\Phi_{\cF}$ on the so-called hybrid spaces, denoted by $X^{\hyb}$. Define the model function $\Phi_{\cF}:X^{\hyb}\to\bR$ by
\begin{align*}
    \Phi_{\cF}(x):=
    \begin{cases}
        \eta(x)\cdot \phi_{\cF}(\psi^{-1}(x)) &\text{ if }x\in\pi_{\hyb}^{-1}(\tau(\bar{\bD}^*_r)),\\
        g_{\cF}(x) &\text{ otherwise}.
    \end{cases}
\end{align*}
By Theorem 2.10 of \cite{Favre:2020rz}, $\Phi_{\cF}$ is continuous.

An important fact for model functions is that they form a dense subset of all the continuous functions on the hybrid space $X^{\hyb}$. Namely, the following result is shown by Favre:
\begin{prop}[\cite{Favre:2020rz}, Theorem 2.12]\label{proposition; hybridK3 - backgrounds - model functions - density}
    Let $\cD(X^{\hyb})$ be the space of all the functions of the form $q\Phi_{\cF}-q'\Phi_{\cF'}$ where $\cF$ and $\cF'$ are regular admissible data and $q$ and $q'$ are positive rational numbers so that $q'\deg(\cF)=q\deg(\cF')$. Then, $\cD(X^{\hyb})$ is a $\bQ$-vector space that is dense in the space of real-valued continuous functions on $X^{\hyb}$ with the topology of the uniform convergence.
\end{prop}

\subsection{Monge-Amp{\`e}re measures}
Let $\pi: X\to\bD^*$ and $\pi_{\cX}:\cX\to\bD$ be as above, i.e., $X$ is an analytic family of projective varieties of dimension $n$ over $\bD^*$ and $\cX$ is its snc model. 

We consider ample line bundles $L_1,\ldots, L_n$ and its semiample model $\cL_1,\ldots,\cL_n$. We see in Section \ref{subsection: hybridK3 - backgrounds - Berkovich geomety} that $X$ and $L_i$ induces a Berkovich variety $X^{\an}_{\bC((t))}$ and its line bundle $L^{\an}_{i,\bC((t))}$ respectively. In this subsection, we define the mixed Monge-Amp{\`e}re measures both on $\cX$ and $X_{\bC((t))}^{\an}$

First, we consider the complex analytic case. Let $h_1,\ldots,h_n$ be smooth metrics on $\cL_1,\ldots,\cL_n$ respectively. Each $h_i$ defines a curvature form denoted by $dd^c \log h_i$, which is locally defined by $dd^c \log h_i\circ s$ for any local section $s$ of $L_i$, respectively. Here, we write
\begin{align*}
    d^c=\frac{1}{i\pi}\left(\partial-\bar{\partial}\right),
\end{align*}
folllowing Demaily \cite{Demailly:1993if}.

We note that curvature forms are defined for smooth metrics, but it extends to continuous psh metrics by means of currents, thanks to Bedford and Taylor \cite{Bedford:1976qd}. That is, if we assume $h_i$'s are continuous psh metrics, i.e., $\log h_i\circ s$ is psh for any local section $s$ of $L_i$, the local Laplacian $dd^c \log h_i\circ s$ still makes sense if we regard $\log h_i\circ s$ as a current, and it is independent of the choice of $s$. That is, they can be patched together to define a global current $dd^c \log h_i$. It is then a closed positive (1,1)-current on $X$. Moreover, since they have locally continuous potentials, we take their wedge product to define a $(n,n)$-form $dd^c \log h_1 \wedge \cdots\wedge dd^c\log h_n$.

For any $t \in \bD$, define  $[X_t]:=dd^c\log|\pi_{\cX}-t|$. For a continuous psh metric $h_1,\ldots,h_n$ on $\cL_1,\ldots,\cL_n$ respectively, we define the mixed Monge-Amp{\`e}re measure associated to $h_1\ldots,h_n$ at $t\in\bD$ by
\begin{align*}
    \MA_t(h_1,\ldots,h_n):=dd^c \log h_1 \wedge\cdots\wedge dd^c \log h_n \wedge [X_t].
\end{align*}
We note that for $t\in\bD^*$, $dd^c\log|\pi_{\cX}-t|$ is the current of integration over $\pi^{-1}\{t\}$, while for $t=0$,
\begin{align*}
    [X_0]=\sum_{E\subset \cX_0}b_E[E],
\end{align*}
where $E$ ranges over all the irreducible components of $\cX_0$ and $b_E$ is the order of zeros of $\pi^*t$ at the generic point of $E$. 

\begin{prop}[\cite{Demailly:1993if}, Corollary 1.6]\label{proposition: backgrounds - continuity of MA measures}
    Let $U\in X$ be a domain and $\{u_1^k\}_{k=1}^{\infty},\ldots,\{u_n^k\}_{k=1}^{\infty}$ be sequences of continuous psh functions converging locally uniformly to $u_1,\ldots,u_n$ respectively. Then, we have a weak convergence of measures
    \begin{align*}
        \lim_{k\to\infty}dd^c u_1^k\wedge\cdots\wedge dd^cu_n^k=dd^c u_1\wedge\cdots\wedge dd^cu_n.
    \end{align*}
\end{prop}

On the other hand, we have a similar theory of currents and psh functions over non-archimedean varieties due to Chambert-Loir and Ducros \cite{chambertloir2012formes}. Since we only use the mixed Monge-Amp{\`e}ere measure for model metrics in the later discussion, we give the exact formula for them as a definition.

We consider models $\cL_1,\ldots,\cL_n$ of $L_1,\ldots,L_n$ over some snc model $\cX$ respectively. Each model $\cL_i$ defines a metric on $L_{i,\bC((t))}^{\an}$, which is called the model metric $g_{\cL_i}$ attached to $\cL_i$. We will not give a definition here and refer to Section 5.3 of \cite{BoucksomEriksson2021}. For such metrics, the mixed Monge-Amp{\`e}ere measure, denoted by $\MA_{\NA}(g_{\cL_1},\ldots,g_{\cL_n})$, is defined as an atomic measure
\begin{align*}
    \MA_{\NA}(g_{\cL_1},\ldots,g_{\cL_n})=\sum_{E\subset\cX_0}(\cL_1|_E\cdot\cdots\cdot\cL_n|_E)\delta_{x_E},
\end{align*}
where $E$ runs over all the irreducible components of $\cX_0$ and $\delta_{x_E}$ denotes a Dirac mass at the divisorial point $x_E$ corresponding to $E$.


For both archimedean and non-archimedean mixed Monge Amp{\`e}re measures, the following continuity is known \cite{BoucksomEriksson2021}:
\begin{prop}[\cite{BoucksomEriksson2021},Theorem 8.4]
    Let $h_i$ be a metric on $\cL_i$ and $g_i$ be a metric on $L_{\bC((t))}^{\an}$. we assume that there exists sequences of metrics $\{h_i^{(j)}\}_{j=1}^{\infty}$ on $\cL_i$ and $\{g_i^{(j)}\}_{j=1}^{\infty}$ on $L_{i,\bC((t))}^{\an}$ such that
    \begin{itemize}
        \item each $g_i^{(j)}$ is a model metric for any $i$ and $j$;
        \item $h_i^{(j)}\to h_i$ and $g_i^{(j)}\to g_i$ respectively as $j\to\infty$, and all the convergence is uniform.
    \end{itemize}
    Then, $\MA_{\NA}(g_1,\ldots,g_n)$ is well-defined as a positive metric on $X_{\bC((t))}^{\an}$ and we have the following weak convergence of measures:
    \begin{align*}
        &\lim_{j\to\infty}\MA_t(h_1^{(j)},\ldots,h_n^{(j)})=\MA_t(h_1,\ldots,h_n);\\
        &\lim_{j\to\infty}\MA_{\NA}(g_1^{(j)},\ldots,g_n^{(j)})=\MA_{\NA}(g_1,\ldots,g_n)
    \end{align*}
\end{prop}

We remark that Theorem 8.4 of \cite{BoucksomEriksson2021} states stronger result than what we present in the non-archimedean case. Namely, they show that the mixed Monge-Amp{\`e}re operator is extended continuously to the set of continuous psh-regularizable metrics. We do not present details about such metrics, but we remark that, by Theorem 7.14 and Definition 7.2 of \cite{BoucksomEriksson2021}, metrics which can be written as a uniform limit of model metrics are psh-regularizable.

\section{Weak convergence of mixed Monge-Amp\`ere measures of model metrics} \label{section: hybridK3 - Weak convergence of mixed Monge-Ampere measures of model metrics}

In this section, we study the weak convergence of mixed Monge-Amp\`ere measures defined from the model metrics of ample line bundles. It is a non-dynamical convergence result similar to one in \cite{Favre:2020rz} and \cite{BHJ1}.

This section considers an analytic family $X\to\bD^*$ of smooth projective varieties of dimension $n$. Let $L_1,\ldots, L_n$ be ample line bundles on $X$. We study the mixed Monge-Amp\`ere measure around the origin.

To study it, we fix an snc model $\cX$ of $X$, a relatively very ample line bundle $\cL$ on $\cX$ with a fixed metric $h$, and semiample models $\cL_1,\ldots,\cL_n$ of $L_1,\ldots,L_n$ respectively as in Section \ref{section: hybridK3 - backgrounds}. We then consider following metrics: a smooth metric $h_i$ of $\cL_i$ and the model metric $g_i$ on $L_{i,\bC((t))}^{\an}$ defined by the model $\cL_i$.

As in the previous section, the mixed Monge-Amp\`ere measure of $(\cL_1,h_1),\ldots,(\cL_n,h_n)$ on $X_t$ for $t\in\bD^*$ is denoted by $\MA_t(h_1,\ldots,h_n)$, and the non-archimedean version of it is denoted by $\MA_{\NA}(g_1,\ldots,g_n)$. We discuss the relation between these two Monge-Amp\`ere measures over the hybrid space. As before, fix a relatively very ample line bundle $\cL$ and let $X^{\hyb}$ be the hybrid analytification of $X$ and $\psi_t:X_t^{\hyb}\to X_t$ and $\psi_0:X_0^{\hyb}\to X_{\bC((t))}^{\an}$ be the natural homeomorphisms.

The following Proposition is a generalization of Theorem 3.4 (and 3.5) of \cite{Favre:2020rz}, where he studies weak convergence of (non-mixed) Monge-Amp\`ere measures attached to model functions.

\begin{prop}\label{prop: hybridK3-weakconvergence-modelmetrics}
    In the above notation, we have the following weak convergence:
    \begin{align*}
        \lim_{t\to0}\psi_{t,*}(\MA_t(h_1,\ldots,h_n))=\psi_{0,*}\MA_{\NA}(g_1,\ldots,g_n).
    \end{align*}
\end{prop}

Since the linear sums of model functions are dense in the set of continuous functions on the hybrid space, it is enough to show the following statement:

\begin{prop}\label{prop: hybridK3-weakconvergence-modelmetrics-modelfunctions}
    Let $\mathcal{F}$ be any regular admissible datum. 
    Then, we have
    \begin{align*}
        \lim_{t\to0}\int_{\psi_t(X_t)}\Phi_{\mathcal{F},t}\psi_{t,*}\MA_t(h_1,\ldots,h_n)=\sum_{E\in\cX_0} g_{\mathcal{F}}(x_E)c_1(\cL_1|_E)\wedge\cdots\wedge c_1(\cL_n|_E).
    \end{align*}
\end{prop}

Denote
\begin{align*}
    \Theta_t&:=\MA_t(h_1,\ldots,h_n) \text{ for }t\in\bar{\bD}_r, \\
    \Xi&:=\MA_{\NA}(g_1,\ldots,g_n).
\end{align*}
We also use the notation
\begin{align*}
    \Theta_E:=c_1(\cL_1|_E)\wedge\cdots\wedge c_1(\cL_n|_E),
\end{align*}
so that
\begin{align*}
    \Xi&=\sum_{E\subset \cX_0}\Theta_E \delta_{x_E}, \text{ and}\\
    \Theta_E&=\int_E \MA_0(h_1,\ldots,h_n).
\end{align*}

\begin{proof}[Proof of Proposition \ref{prop: hybridK3-weakconvergence-modelmetrics-modelfunctions}]
We follow the argument in \cite{Favre:2020rz}.

Let $\cX'$ be an snc model of $X$ attached to the model $\cF$. Taking a model dominating $\cX$ and $\cX'$ if necessary, we may assume that $\cX=\cX'$. Take a compact neighborhood $K$ of the singular locus of the special fiber $\cX_0$. 
We use the following fact shown by Favre.
\begin{lmm}[\cite{Favre:2020rz}, Lemma 3.6]\label{lmm: hybridK3-weakconvergence-modelmetrics-modelfunctions-continuousextension}
    Let $\cF$ be any regular admissible datum and $K$ be a compact set taken above, and set $\varphi:=\Phi_{\cF}\circ\psi=(\log r\cdot \phi_{\cF}) / \log|t|$. The function $\varphi:X\to\bR$ extends continuously to $\cX\setminus K$.
    Moreover, For any irreducible component $E$ of $\cX_0$, $\varphi(x)=g_{\cF}(x_E)$ for any $x\in E\setminus K$, where $x_E$ is a divisorial point associated to $E$ in $X^{\an}_{\bC((t))}$. 
\end{lmm}
Set
\begin{align*}
    \Delta_t:=\left| \int_{X_t}\Phi_{\cF}d\psi_{*,t}\Theta_t-\sum_{E\subset\cX_0}g_{\cF}(x_E)\Theta_E\right|
\end{align*}
and it is enough to show that $\limsup_{t\to0}\Delta_t=0$. 
\begin{align*}
    \Delta_t &= \left| \int_{X_t}\Phi_{\cF}d\psi_{*,t}\Theta_t-\sum_{E\subset\cX_0}g_{\cF}(x_E)\Theta_E\right| \\
    &= \left| \int_X \Phi_{\cF}\circ\psi d\Theta_t-\sum_{E}g_{\cF}(x_E)\Theta_E \right| \\
    &=\left| \int_{X_t} \varphi d\Theta_t-\sum_{E}g_{\cF}(x_E)\Theta_E \right| \\
    &\leq \left(\left| \int_K \varphi d\Theta_t \right| + \left| \int_{X_t\setminus K} \varphi d\Theta_t - \sum_E g_{\cF}(x_E)\Theta_E \right| \right). \\
\end{align*}
We estimate the second term of the last line.
Take any $\epsilon>0$. Note that the map
\begin{align*}
    t\mapsto \MA_t(h_1,\ldots,h_n)
\end{align*}
is continuous with respect to weak topology. Hence, together with Lemma \ref{lmm: hybridK3-weakconvergence-modelmetrics-modelfunctions-continuousextension}, we see that there exists a small positive number $r_0>0$ such that, for any $t\in\bar{\bD}_{r_0}$,
\begin{align*}
    \left| \int_{X_t\setminus K} \varphi d\Theta_t - \sum_E g_{\cF}(x_E)\Theta_E \right|&<\epsilon+\sum_{E\subset\cX_0}g(x_E)\int_{E\cap K}\MA_0(h_1,\ldots,h_n) \\
\end{align*}

Hence it is enough to show that for any $\epsilon>0$, there exists a small compact neighborhood $K$ and a real number $r_1>0$ such that for any $t\in\bar{\bD}_{r_1}$,
\begin{align*}
    \max\left\{\left| \int_K \varphi d\Theta_t \right|, \left|\int_{E\cap K}\MA_0(h_1,\ldots,h_n)\right|\right\}<\epsilon.
\end{align*}

Note that $\Phi_{\cF}$ is continuous on the hybrid space $X^{\hyb}$. The topological closure $\overline{\psi(K\setminus\{t=0\})}$ in $X^{\hyb}$ is a compact set and hence $\Phi_{\cF}$ is bounded, from which we can deduce that $\Phi_{\cF}\circ\psi$ is bounded in $K$. Hence, it is enough to show that 
\begin{align*}
    \left| \int_K \varphi d\Theta_t \right|<\epsilon.
\end{align*}
Since $E\cap Z$ is a proper closed subvariety for any irreducible component $E$ of $\cX_0$, we have
\begin{align*}
    \int_{E\cap Z}\MA_0(h_1,\ldots,h_n)|_E=0,
\end{align*}
which implies
\begin{align*}
    \MA_0(h_1,\ldots,h_n)(Z)=0.
\end{align*}
By continuity, we have
\begin{align*}
    \MA_t(h_1,\ldots,h_n)(K)< \epsilon
\end{align*}
for a sufficiently small compact neighborhood $K$ of $Z$ and $t$, i.e.,
\begin{align*}
    \left| \int_K \varphi d\Theta_t \right|<\epsilon.
\end{align*}
Now, for such $r_1$ and $K$, set $r_2:=\min(r_0,r_1)$. For any $t\in^bar{\bD}_{r_2}$,
\begin{align*}
    \Delta_{t}&\leq \left(\left| \int_K \varphi d\Theta_t \right| + \left| \int_{X_t\setminus K} \varphi d\Theta_t - \sum_E g_{\cF}(x_E)\Theta_E \right| \right) \\
    &\leq \epsilon + \epsilon + \sum_{E\subset{\cX_0}}g(x_E)\epsilon \\
    &=\left(2+\sum_{E\subset \cX_0}g(x_E)\right)\epsilon.
\end{align*}
Since $2+\sum g(x_E)$ is independent of $t$ and $K$ and $\epsilon$ can be taken arbitrarily small, we complete the proof.
\end{proof}

\section{hybrid dynamics of K3 surfaces}\label{section: hybridK3 - main theorem}


\subsection{Construction of $\eta_t$} \label{subsec: main theorem - construction of the inv measures and main theorem}
We follow the notation from Section \ref{subsection: background - hybrid K3 dynamics}. Before stating our result, we prepare setups. As we saw in Section \ref{subsection: background - hybrid K3 dynamics}, the invariant measures attached to hyperbolic automorphisms of K3 surfaces are constructed by Cantat \cite{Cantat:2001} in the complex case and Filip \cite{Filip:2019ij} in the non-archimedean case in slightly different ways. We re-construct them following Filip's construction in both cases to argue them over the hybrid setting.

As in Section \ref{subsection: background - hybrid K3 dynamics}, let $X$ be an analytic family of projective K3 surfaces and $f: X\to X$ be that of hyperbolic automorphisms. Let $L_1,\ldots, L_k$ be ample line bundles over $X$ such that $L_{1,K}.\ldots,L_{k,K}$ form a basis of $\Pic(X_K)$.

Take an snc model $\cX$ of $X$ and $\cL_1,\ldots,\cL_k$ be semiample $\bQ$-line bundles such that $\cL_i|_X=L$. For each $n\in\bZ$, we take an snc model $\cX^{(n)}$ and a morphism $f'^{(n)}:\cX^{(n)}\to\cX$ such that $f'^{(n)}|_X=f^n$.

We fix a very ample line bundle $\cL$ of $\cX$. For any $n\in\bZ$, We note that, as $L_1,\ldots,L_k$ generate the Picard group on $X_K$, there exist natural numbers $n_1,\ldots,n_k$ and a vertical Cartier divisor $D$ such that
\begin{align*}
    \cL\simeq \bigotimes \cL_i^{\otimes n_i}\otimes \cO_{\cX}(D).
\end{align*}
For the later argument, we set $\cL^{(n)}$ to be the pull-back $(f'^{(n)})^*\cL$.

We observe that since $f$ is automorphism, there exists an matrix $A=(a_{ij})\in\GL(n,\bZ)$ such that
\begin{align*}
    f^*L_i\simeq \bigotimes_j L_j^{\otimes a_{ij}} \text{ for }i=1,\ldots,k,
\end{align*}
whence
\begin{align}\label{eq: hybridK3-weakconvergence-equation-of-linebundles-under-the-action-of-f}
    f^*L_i^{\otimes m}\simeq \bigotimes L_j^{\otimes ma_{ij}} \text{ for }i=1,\ldots,k.
\end{align}
For the later argument, set $(a_{ij}^{(n)}):=A^n$ for any $n\in\bZ$.

Define a $\bG_m^k$-torsor $E$ over X as
\begin{align*}
    E:=\bigoplus_i (L_i^{\otimes m})^*,
\end{align*}
where $L^*$ for any line bundle $L$ is a $\bG_m$-torsor given by removing the zero sections from $L$. By the relation (\ref{eq: hybridK3-weakconvergence-equation-of-linebundles-under-the-action-of-f}), we have an equivalence of torsors $E\simeq f^*E$. By combining the natural morphism $f^*E\to E$, we have a morphism $F: E\to E$ that is compatible with $f: X\to X$. We note that we can write the map $F$ in a local chart using $A$. That is, for $z\in X_t$, take open neighborhoods $z\in U$ and $f(z)\in V$ of $X_t$ such that $f(V)\subset U$ and they trivialize $E$. Then, for any $(z,(s_i))\in U\times (\bC^*)^k$, $F(z,(s_i))=(f(z),(\prod_js_j^{a_ij})$. For the later argument, set
\begin{align*}
    A^n=(a_{ij}^{(n)})_{i,j}
\end{align*}
so that
\begin{align*}
    (f^n)^*L_i\simeq \bigotimes_j L_j^{\otimes a_{ij}^{(n)}}.
\end{align*} 

Since $L=\sum_in_iL_i$, we have a natural homomorphism $\varphi_0:E\to (L^{\otimes m})^*$. In a local chart, $\varphi_0$ is described as
\begin{align*}
    \varphi_0(z,(s_i)_i)=\left(z, \prod_{i=1}^k s_i^{n_i}\right).
\end{align*}
In the same way, for any $n\in \bZ$ there is a natural homomorphism $\varphi_n:E\to (L_n^{\otimes m})^*$, so that $f^n\circ \varphi_n=\varphi_0\circ F^n$. 

First, we represent $\eta_t$ for $t\in \bD^*$ as a weak limit of the mixed Monge-Amp{\`e}re measures $\{\eta_{n,t}\}$. Define a sequence of functions $\{G_n:E\to\bR\}_{n=-\infty}^{\infty}$ by
\begin{align*}
    G_n:=\frac{1}{m\cdot \lambda^n}\log\left( (F^n)^*\varphi^*_0h^m\right).
\end{align*}
For each $n=1,2,3,\ldots$, we attach a family of measure $\{\eta_{n,t}\}_{t\in \bD^*}$ as follows: for any open subset $U\subset X_t$ and non-vanishing section $s:U\to E$, define $\eta_{n,t}|_U:=dd^c (G_n\circ s)\wedge dd^c (G_{-n}\circ s)$. 

\begin{lmm}
    The measure $\eta_{n,t}$ is well-defined.
\end{lmm}
\begin{proof}
    It is enough to show that $dd^c (G_n\circ s)\wedge dd^c (G_{-n}\circ s)$ is independent of the choice of $s$. Let $s'$ be another section on $U'$ such that $U\cap U'\neq\empty$ and $s'=us$ on $U\cap U'$ for some $u:U\cap U'\to (\bC^*)^n$. If we write $s(z)=(z, (s_i(z))_i)$ and $u(z)=(u_i(z))$ in a local chart, 
    \begin{align*}
        G_n\circ s'(z)&=\frac{1}{m\cdot \lambda^n}\log\left( (F^n)^*\varphi^*_0h^m(s'(z))\right) \\
        &=\frac{1}{m\cdot \lambda^n}\log\left( (F^n)^*\varphi^*_0h^m(us(z))\right) \\
        &=\frac{1}{m\cdot \lambda^n}\left(\log\left|\prod_i u_i(z)^{\sum_j a_{ij}^{(n)}n_j} \right|+\log\left( (F^n)^*\varphi^*_0h^m(s(z))\right)\right) \\
        &=\frac{1}{m\cdot \lambda^n}\left(\sum_j a_{ij}^{(n)}n_j \log|u_i(z)|+G_n\circ s(z)\right).
    \end{align*}
    Since $\log|u_i|$ is harmonic for each $i$, the first and last lines must coincide after taking $dd^c$.
\end{proof}
We set $h_n':=(f'^{(n)})^*h'$ and $h_n:=(f^n)^*h$. $h_n'$ is a smooth metric on $\cL^{(n)}$ and $h_n$ is on $(f^n)^*L$. Since $f'^{(n)}$ is extension of $f^n$ to some snc model $\cX^{(n)}$ of $X$, $\cL^{(n)}|_X=(f^n)^*L$ and $h_n'|_{(f^n)^*L}=h_n$. We observe that $\eta_{n,t}$ is given from these metrics, too:
\begin{prop}
    For any $n=1,2,3,\ldots$ and $t\in\bD^*$,
    \begin{align*}
        \eta_{n,t}=\frac{1}{m^2}\MA_t(h_n^m,h_{-n}^m).
    \end{align*}
\end{prop}
\begin{proof}
    Take any open set of $U\subset X_t$ and a section $s: U\to E$. We note that $\varphi_n(s):U\to L_n$ is a non-vanishing section of $L_n$. By the definition of the mixed Monge-Amp{\`e}re operator and $\eta_{n,t}$, it is enough to compere $dd^c (G_n\circ s)\wedge dd^c (G_{-n}\circ s)$ and $dd^c \log|h_n(\varphi_n(s))|\wedge dd^c \log|h_{-n}(\varphi_{L_{-n}}(s))|$. Since $f^n\circ \varphi_n=\varphi_0\circ F^n$,
    \begin{align*}
        G_n(s)&=\frac{1}{m\lambda^n}\log|h^m(\varphi_0(F^n(s)))| \\
        &=\frac{1}{m\lambda^n}\log|h^m(f^n(\varphi_n(s)))| \\
        &=\frac{1}{m\lambda^n}\log|h^m_n(\varphi_n(s)))|.
    \end{align*}
    Therefore, for $n=1,2,3,\ldots$,
    \begin{align*}
        dd^c (G_n\circ s)\wedge dd^c (G_{-n}\circ s)&=dd^c \frac{1}{m\lambda^n}\log|h^m_n(\varphi_n(s)))|\wedge dd^c \frac{1}{m\lambda^{-n}}\log|h^m_{-n}(\varphi_{-n}(s)))| \\
        &=\frac{1}{m^2}dd^c \log|h_n(\varphi_n(s))|\wedge dd^c \log|h_{-n}(\varphi_{-n}(s))|,
    \end{align*}
    which shows the claim.
\end{proof}

In Section \ref{subsec: hybridK3-weakconvergence-invariantmeasure-uniformestimation}, we will show that there exists continuous functions $G_{\pm\infty}:E\to\bR$ such that
\begin{align*}
    \lim_{n\to\infty} G_n&=G_{\infty}, \\
    \lim_{n\to-\infty} G_n&=G_{-\infty},
\end{align*}
where the limit is uniform on each $X_t$. We remark that this limit has uniformity in $t$, which is essential in our proof of the main theorem. So far, we admit this convergence and define $\eta_t$ as the weak limit of $\{\eta_{n,t}\}$ for any fixed $t\in\bD^*$. We note that $\eta_t$ is locally written as $dd^c (G_{\infty}\circ s)\wedge dd^c (G_{-\infty}\circ s)$ for a local section $s$. By definition of $G_{\pm\infty}$,
\begin{align*}
    G_{\pm\infty}\circ F=\lambda^{\pm 1}G_{\pm\infty},
\end{align*}
so that $\eta_t$ is invariant under the action of $f^*$. We note that by the uniqueness in Proposition \ref{proposition: hybridK3 - canocurrents for cpx K3 surfaces}, this $\eta_t$ coincides with the one constructed by Cantat.

We observe that the total mass of $\eta_t$ is independent of $t$. Recall that $\eta_{n,t}$ is written as the mixed Monge-Amp{\`e}re mesure, whose total mass is calculated by the intersection number. Since $L^{(n)}$ is written by a linear sum of $L_1,\ldots,L_k$, it is determined from the intersection number $(L_1\cdots L_n)$, which is independent of $t$. The total mass of their limit is, hence, also independent of $t$. We set $M:=\eta_t(X_t)$ to be the total mass.

Next, we discuss the non-archimedean version of the above argument. Set $g$ to be the model metric on $L_{\bC((t))}^{\an}$ attached to $\cL$ and $g_n$ to be that on $(f_{\bC((t))}^{\an,n})^*L_{\bC((t))}^{\an}$ attached to $\cL^{(n)}$ for any $n\in\bZ$. By definition, $g_n=(f_{\bC((t))}^{\an, n})^*g$. We note that all the objects and morphisms appearing above, i.e., the variety $X$, the line bundles $L$ and $L_i$, the $\bG_m^k$-torsor $E$, and the morphisms $f$, $F$, and $\phi_n$, have the non-archimedean analytification. 

We define the same function as $G_{n}$ and $G_{\pm\infty}$ in the non-archimedean setting; set 
\begin{align*}
    G_n^{\NA}=\frac{1}{m\cdot\lambda^n}\log \left((F_{\bC((t))}^{\an})^* (\varphi_{L,\bC((t))}^{\an})^*g^m\right).    
\end{align*}

As in the complex case, we can define $\eta_n^{\NA}$, which coincides with $\MA_{\NA}(g_n^m,g_{-n}^m)$. For details, see Section 5.6 of \cite{Filip:2019ij}.

We note that, since $g_n^m$ is a model metric associated to $(\cL_n^m,L_n^m)$ and $h_n^m$ is a smooth metric on $L_n^m$, Proposition \ref{prop: hybridK3-weakconvergence-modelmetrics} implies the following:

\begin{prop}\label{prop: construction of measures - weak conv of nth measures}
    For each $n=1,2,3,\ldots$, the following weak convergence of measures hold:
    \begin{align*}
        \lim_{t\to 0}\psi_{t,*}\eta_{n,t}=\psi_{0,*}\eta_n^{\NA}.
    \end{align*}
\end{prop}

Especially, the total mass of $\eta^{\NA}_n$ coincides with that of $\eta_{n,t}$, which is independnt of $t$.

Filip \cite{Filip:2019ij} showed that the sequence of functions $G_{n}^{\NA}$ uniformly converges to $G_{\infty}^{\NA}$ as $n\to\infty$ and $G_{-\infty}^{\NA}$ as $n\to-\infty$. Then, Corollaire 5.6.5 of \cite{chambertloir2012formes}, i.e., the non-archimedean version of Proposition \ref{proposition: backgrounds - continuity of MA measures} shows that there exists a measure $\eta_0$ such that $f^*\eta_0=\eta_0$, which is a weak limit of $\{\eta_n^{\NA}\}$. For details, see Corollary 5.6.8 of \cite{Filip:2019ij}.

We note that $\eta_0$ has total mass $M$ since that of $\eta_n^{\NA}$ coincides $\eta_t(X_t)$, whose limit is $M$. In Section \ref{section: hybridK3 - Introduction}, we assume that the measures appearing in Theorem \ref{thm: hybridK3-weakconvergence-invariantmeasure} is a probability measure. Since the total mass of $\eta_t$ and $\eta_0$ are all the same, it is enough to show the following statement:
\begin{thm}\label{thm: hybridK3 - main theorem - main theorem}
    In the above notation, we have a weak convergence
    \begin{align*}
        \lim_{t\to0}\psi_{t,*}\eta_t=\psi_{0,*}\eta_0.
    \end{align*}
\end{thm}
As in the proof of Proposition \ref{prop: hybridK3-weakconvergence-modelmetrics}, it is enough to take test functions as model functions, which reduces the theorem to the following statement:
\begin{thm}\label{thm: hybridK3-weakconvergence-invariantmeasure-modelfunctions}
    Let $\cF$ be any regular admissible datum. Then, we have
    \begin{align*}
        \lim_{t\to0}\int_{X_t} \Phi_{\cF,t}d\eta_t=\int_{X_{\bC((t))}^{\an}}g_{\cF}d\eta_0.
    \end{align*}
\end{thm}
Here, we sketch the proof. We follow the argument by Favre in \cite{Favre:2020rz}. Since
\begin{align*}
    &\left|\int_{X_t}\Phi_{\cF}d\eta_t-\int_{X_{\bC((t))}^{\an}}g_{\cF}d\eta_0\right| \\
    &\leq\left| \int_{X_t}\Phi_{\cF}d\eta_t-\int_{X_t}\Phi_{\cF}d\eta_{n,t} \right|+\left|\int_{X_t}\Phi_{\cF}d\eta_{n,t}-\int_{X_{\bC((t))}^{\an}}g_{\cF}d\eta_{n,0}\right|+\left|\int_{X_{\bC((t))}^{\an}}g_{\cF}d\eta_{n,0}-\int_{X_{\bC((t))}^{\an}}g_{\cF}d\eta_0\right|,
\end{align*}
we divide the estimation into the above three parts, denoted by $I_1$, $I_2$, and $I_3$, each term of the right-hand side of the above inequality, respectively. Note that Proposition \ref{prop: construction of measures - weak conv of nth measures} shows $I_2\to 0$. Also, the convergence of $I_3$ follows from the weak convergence of $\eta_{n,0}$ to $\eta_0$, which is Corollary 5.6.8 of \cite{Filip:2019ij}. Hence, essentially, it is enough to estimate $I_1$, where we need uniformity on $t$.

We first show the uniform convergence of $\{G_n\}$ in Section \ref{subsec: hybridK3-weakconvergence-invariantmeasure-uniformestimation}. We complete the proof of Theorem \ref{thm: hybridK3-weakconvergence-invariantmeasure-modelfunctions} in Section \ref{subsec: hybridK3-weakconvergence-invariantmeasure-proof}.

\subsection{Estimation of potentials for archimedean parameters} \label{subsec: hybridK3-weakconvergence-invariantmeasure-uniformestimation}
We use the same notation as the previous subsection. The goal is the estimation of $\{G_n\}$ as $n\to\pm\infty$. The basic argument parallels Filip's one \cite{Filip:2019ij} in the non-archimedean case, but with a more detailed estimation to study the dependence on $t$.

We prepare several more notations. Take a Zariski open covering $\{U_{i,K}\}_{i=1}^{N}$ of $X_K$ so that $E_K$ has a non-vanishing section for each $U_{i,K}$. Then, by shrinking the parameter disk if necessary, there is an open covering $\{U_i\}_{i=1}^{N}$ of $X$ with a non-vanishing section $s_i$ of $E$ on $U_i$.

The following Proposition is the critical estimate in the proof of the main theorem. This statement corresponds to Theorem 4.2 of \cite{Favre:2020rz} in case of degeneration of rational functions, but the estimation here is more delicate.

\begin{prop}\label{prop: hybridK3-weakconvergence-invariantmeasure-uniformestimation}
    In the above notation, there exists a decreasing sequence of positive real numbers $\{\epsilon_n\}$ converging to 0 such that
    \begin{align*}
        \sup_{z\in X_t, i=0,\ldots,N}\left|G_{\pm\infty, t}(s_i(z))-G_{\pm n,t}(s_i(z))\right|<\epsilon_n\log|t|^{-1}
    \end{align*}
for any sufficiently small $t$.
\end{prop}
To show Proposition \ref{prop: hybridK3-weakconvergence-invariantmeasure-uniformestimation}, we first show the following lemma:

\begin{lmm}\label{lmm: hybridK3-weakconvergence-invariantmeasure-uniformestimation-cohomology}
    In the above notation i.e., $(n_i)$ is a coordinate of $\cL$ with respect to $\cL_1,\ldots,\cL_k$ and $A=(a_{ij})$ is the representation matrix of $f^*$ with respect to the same basis,
    \begin{align*}
        \lim_{n\to\infty}\frac{1}{\lambda^n}A^n(n_i)^t=v_f,
    \end{align*}
    where $v_f$ is an eigenvector of $A$. Moreover, for each $n$ sufficiently large,
    \begin{align*}
        \left\|\frac{1}{\lambda^{n+1}}A^{n+1}(n_i)^t-\frac{1}{\lambda^n}A^n(n_i)^t\right\|<\frac{C_1}{\delta^n}
    \end{align*}
    for some $\delta>0$ and $C_1>0$, where $\|\cdot\|$ is any norm on $\bR^k$.
\end{lmm}
Note that this estimation is completely independent of $t$.
\begin{proof}
    Working on $\bC$, we may assume that there exists a matrix $U\in \GL(n,\bC)$ such that
    \begin{align*}
        &B:=UAU^{-1}=\diag[\lambda,\lambda^{-1},c_3,\ldots,c_k],\text{ and }\\
        &v_f'=Uv_f=(a,0,\ldots,0)^t,
    \end{align*}
    where $c_3,\ldots c_k$ are eigenvalues of $A$ whose absolute values are all $1$ and $a\in\bC$. Denote $u=(n_i)^t$ and $u'=Uu=(u_1',\ldots,u_k')^t$. We may assume that $a=u_1'$. Then, 
    \begin{align*}
        \frac{1}{\lambda^n}B^nu'= \left(u_1',\frac{u_2'}{\lambda^{2n}},\frac{c_3u'_3}{\lambda^n},\ldots,\frac{c_ku_k'}{\lambda^n}\right)^t,
    \end{align*}
    whence
    \begin{align*}
        \left\|\frac{1}{\lambda^n}B^{n+1}u'-\frac{1}{\lambda^n}B^nu'\right\|&<\left\|\left(u_1',\frac{u_2'}{\lambda^{2(n+1)}},\frac{c_3u'_3}{\lambda^{n+1}},\ldots,\frac{c_ku_k'}{\lambda^{n+1}}\right)^t-\left(u_1',\frac{u_2'}{\lambda^{2n}},\frac{c_3u'_3}{\lambda^n},\ldots,\frac{c_ku_k'}{\lambda^n}\right)^t\right\| \\
        &=\frac{1}{\lambda^n}\left\|\left(0,-\frac{u_2'}{\lambda^{n}}\left(1-\frac{1}{\lambda^2}\right),-c_3u'_3\left(1-\frac{1}{\lambda}\right),\ldots,-c_ku_k'\left(1-\frac{1}{\lambda}\right)\right)^t\right\|,
    \end{align*}
    and we have the claim by changing the coordinate by $U$.
\end{proof}

\begin{proof}[Proof of Proposition \ref{prop: hybridK3-weakconvergence-invariantmeasure-uniformestimation}]
    It is enough to show the inequality for $G_{\infty}$, and the proof for $G_{-\infty}$ is parallel to it if we replace $f$ by $f^{-1}$. 
    Note that by putting $A^n=(a_{ij}^{(n)})$, for each $n$, $u=(u_m)\in\bC^k$ and $z\in E$ we have
    \begin{align*}
        G_n(uz)= G_n(z)+ \frac{1}{\lambda^n}\sum_{l,m} a_{ml}^{(n)} n_l \log|u_m|,
    \end{align*}
    Using this relation, we define the function $\alpha_n: (\bC^*)^k\to\bC^*$ by
    \begin{align*}
        \alpha_n((u_m)):=\prod_m u_m^{\sum_{l,m} a_{ml}^{(n)} n_l},
    \end{align*}
    so that for each $n$, we have
    \begin{align*}
        G_n(uz)=G_n(z)+\frac{1}{\lambda^n}\log\alpha_n(u).
    \end{align*}
    For $x\in U_i\cap f^{-1} U_j$,
    \begin{align*}
        Fs_i(x)=u_{ij}(x)s_j(f(x))
    \end{align*}
    for some $u_{ij}:U_i\cap f^{-1}(U_j) \to(\bC^*)^k$. 

    Now estimate $G_n$. We have
    \begin{align*}
        G_{n+1}(s_i(x))&=\frac{1}{\lambda}F^*G_n(s_i(x)) \\
        &=\frac{1}{\lambda}G_n(u_{ij}(x)s_j(f(x))) \\
        &=\frac{1}{\lambda}\left(G_n(s_j(x))+\frac{1}{\lambda^n}\log|\alpha_n(u_{ij}(f(x)))|\right) \\
        &=\frac{1}{\lambda}\left(G_n(s_j(x))+\frac{1}{\lambda^n}\sum_{l,m}a_{ml}^{(n)}n_l\log|u_{ij,m}|\right)
    \end{align*}
    where $u_{ij}=(u_{ij,1},\ldots,u_{ij,k})$.
    We estimate $u_{ij}$ around the central fiber. Note that the map $F$ is algebraic for each fixed $t$ and moves analytically on $t$. Take open coverings $\{\cU_i\}$ of $\cX$ and $\{\cU_i'\}$ of $\cX'$ so that $\cU_i\cap X=\cU_i'\cap X=U_i$. 
    
    For each point $\xi\in\cU_i'\cap f'^{-1}(\cU_j)$, there exists an open neighborhood $\cU\subset\cU_i'\cap f'^{-1}(\cU_j)$ and its local coordinate $x=(x_1,x_2,x_3)$ so that the given point $\xi$ is the origin and the central fibre is written as $x_1^{l_1}x_2^{l_2}x_3^{l_3}=0$ for $l_1,l_2,l_3\in \bZ_{\geq 0}$. Then, we can write
    \begin{align*}
        u_{ij,m}(x)=x_1^{p_{1,m}}x_2^{p_{2,m}}x_3^{p_{3,m}}v_{ij,m}(x)
    \end{align*}
    for some $p_{1,m},p_{2,m},p_{3,m}\in\bZ$ on $\cU\cap X$ so that $p_{i,m}=0$ when $l_i=0$, and $v_{ij,m}$ is holomorphic at $\xi$ and nonzero. 
    
    By taking $\cU$ small enough, hence, we may assume that for some $M_1>0$
    \begin{align*}
        M_1^{-1}<|v_{ij,m}|<M_1,
    \end{align*}
    i.e., by taking the logarithm
    \begin{align*}
        -\log M_1<\log|v_{ij,m}|<\log M_1.
    \end{align*}
    Let $M_2\in\bZ$ be so large that
    \begin{align*}
        \bigg|p_{1,m}\log|z_1|+p_{2,m}\log|z_2|+p_{3,m}\log|z_3|\bigg|\leq M_2\log|t|^{-1}.
    \end{align*}
    Then,
    \begin{align*}
        &\left|\frac{1}{\lambda^{n+1}}\log|\alpha_{n+1}(u_{ij}(f(x))|-\frac{1}{\lambda^n}\log|\alpha_n(u_{ij}(f(x))|\right|\\
        &\leq \left|\sum_{l,m}\left(\frac{1}{\lambda^{n+1}}a_{ml}^{(n+1)}-\frac{1}{\lambda^n}a_{ml}^{(n)}\right)(n_l\log|u_{ij,m}|)\right| \\
        &=\left|\sum_{l,m}\left(\frac{1}{\lambda^{n+1}}a_{ml}^{(n+1)}-\frac{1}{\lambda^n}a_{ml}^{(n)}\right)(n_l(p_{1,m}\log|z_1|+p_{2,m}\log|z_2|+p_{3,m}\log|z_3|+\log|v_{ij,m}|))\right| \\
        &\leq \sum_{l,m}\left(\frac{1}{\lambda^{n+1}}a_{ml}^{(n+1)}-\frac{1}{\lambda^n}a_{ml}^{(n)}\right)(n_l(M_2\log|t|^{-1}+\log|v_{ij,m}|)).
    \end{align*}
    Since $\log|v_{ij,m}|$ is bounded, Lemma \ref{lmm: hybridK3-weakconvergence-invariantmeasure-uniformestimation-cohomology} implies
    \begin{align*}
        \left|\frac{1}{\lambda^{n+1}}\log|\alpha_{n+1}(u_{ij}(f(x))|-\frac{1}{\lambda^n}\log|\alpha_n(u_{ij}(f(x))|\right|\leq\frac{C_2}{\delta^n}\log|t|^{-1}
    \end{align*}
    for some $C_2>0$. Put
    \begin{align*}
        d_n(t):=\sup_{\substack{i=1,\ldots,N\\ x\in U_i}}\bigg|G_n(s_{i,t}(x))-G_{n-1}(s_{i,t}(x))\bigg|.
    \end{align*}
    Then, since
    \begin{align*}
        |G_{n+1}(s_{i,t}(x))-G_n(s_{i,t}(x))|&= \frac{1}{\lambda}|G_n(F(s_{i,t}(x)))-G_{n-1}(F(s_{i,t}(x)))| \\
        &\leq \frac{1}{\lambda}|G_n(s_j(f(x)))-G_{n-1}(s_j(f(x)))|+\left|\frac{1}{\lambda^{n+1}}\log|\alpha_n(x)|-\frac{1}{\lambda^{n}}\log|\alpha_{n-1}(x)|\right| \\
        &\leq \frac{1}{\lambda}d_{n}(t)+\frac{1}{\lambda}\frac{C_2}{\delta^n}\log|t|^{-1}
    \end{align*}
    for any $i$ and $x$, we have
    \begin{align*}
        d_{n+1}(t)\leq \frac{1}{\lambda}d_n(t)+\frac{1}{\lambda}\cdot\frac{C_2}{\delta^n}\log|t|^{-1}.
    \end{align*}
    Therefore we have
    \begin{align*}
        d_n(t)&\leq \frac{1}{\lambda}\cdot\frac{1}{\delta^{n-1}}\log|t|^{-1}+\cdots\frac{1}{\lambda^{n-1}}\cdot\frac{1}{\delta}\log|t|^{-1}+\frac{1}{\lambda^{n-1}}d_1(t) \\
        &\leq \frac{(n-1)C_2}{\lambda^n}\log|t|^{-1}+\frac{1}{\lambda^{n-1}}d_1(t) \\
        &\leq \frac{(n-1)C_2}{\lambda^n}\log|t|^{-1}+\frac
        {1}{\lambda^{n-1}}\frac{C_2}{\delta}\log|t|^{-1},
    \end{align*}
    where we assume $\lambda\geq\delta$ (replace $\lambda$ by $\delta$ if not). 
    
    Since $\sum_n (n-1)/\lambda^n$ and $\sum_n 1/\lambda^n$ converges, we have
    \begin{align*}
        \sum_n d_n(t)\leq C_2\cdot\sum_{n=1}^{\infty}\left(\frac{n-1}{\lambda^n}+\frac{1}{\delta\cdot\lambda^{n-1}}\right)\log|t|^{-1}.
    \end{align*} 
    We set
    \begin{align*}
        \epsilon_n:=C_2\cdot\sum_{i=n+1}^{\infty}\left(\frac{n-1}{\lambda^n}+\frac{1}{\delta\cdot\lambda^{n-1}}\right)
    \end{align*}
    to complete the statement.
\end{proof}

\subsection{Proof of Theorem \ref{thm: hybridK3-weakconvergence-invariantmeasure-modelfunctions}}\label{subsec: hybridK3-weakconvergence-invariantmeasure-proof}
We continue to use notations in the previous subsections. Following the sketch of the proof in Section \ref{subsec: main theorem - construction of the inv measures and main theorem}, we estimate 
\begin{align*}
    \left|\int_{X_t}\Phi_{\cF}d\eta_t-\int_{X_t}\Phi_{\cF}d\eta_{n,t}\right|.
\end{align*}
Note that $\eta_t=dd^c s_i^*G$ on $W_i$ by definition and $\Phi_{\cF}=\log r\phi_{\cF}/\log|t|^{-1}$. Take a partition of unity and work on $W_i$ (or a smaller set if necessary) to get
\begin{align*}
    \int_{W_{i,t}}\phi_{\cF}d\eta_t&=\int_{W_{i,t}}\phi_{\cF}dd^c s_i^*G_{n,t}^+\wedge dd^c s_i^*G_{n,t}^- \\
    &=\int_{W_{i,t}}s_i^*G_{n,t}^+dd^c\phi_{\cF}\wedge dd^c s_i^*G_{n,t}^-,
\end{align*}
and the same for the second integral. Hence 
\begin{align*}
    &\left|\int_{W_{i,t}}\Phi_{\cF}d\eta_t-\int_{W_{i,t}}\Phi_{\cF}d\eta_{n,t}\right| \\
    &=\left|\frac{\log r}{\log|t|^{-1}}\int_{W_{i,t}}\phi_{\cF}d\eta_t-\int_{W_{i,t}}\phi_{\cF}d\eta_{n,t}\right| \\
    &= \frac{|\log r|}{\log|t|^{-1}}\left|\int_{W_{i,t}}s_i^*G_{n,t}dd^c\phi_{\cF}\wedge dd^c s_i^*G_{n,t}^- -\int_{W_{i,t}}s_i^*G_{t}^+dd^c\phi_{\cF}\wedge dd^c s_i^*G_{t}^-\right| \\
    &= \frac{|\log r|}{\log|t|^{-1}}\left|\int_{W_{i,t}}(s_i^*G_{n,t}-s_i^*G_t^+ )dd^c\phi_{\cF}\wedge dd^cs_i^*G_{n,t}^- + \int_{W_{i,t}} (s_i^*G_{n,t}^--s_i^*G_t^- )dd^c \phi_{\cF} \wedge dd^c s_i^*G_t^+ \right| \\
    &\leq \epsilon_n|\log r|\left|\int_{W_{i,t}}\left(dd^c\phi_{\cF}\wedge dd^cs_i^*G_{n,t}^-+dd^c \phi_{\cF} \wedge dd^cs_i^* G_t^+ \right)\right|,
\end{align*}
where the last inequality follows from Proposition \ref{prop: hybridK3-weakconvergence-invariantmeasure-uniformestimation}. Note that the right-hand side of the last inequality is independent of $i$, and the measures inside of the integral are all positive. Therefore,
\begin{align*}
    \left|\int_{X_t}\Phi_{\cF}d\eta_t-\int_{X_t}\Phi_{\cF}d\eta_{n,t}\right|\leq \epsilon_n|\log r|\left|\int_{X_t}\left(dd^c\phi_{\cF}\wedge dd^cs_i^*G_{n,t}^-+dd^c \phi_{\cF} \wedge dd^cs_i^* G_t^+ \right)\right|.
\end{align*}
The integral inside of the absolute value can be calculated in terms of cohomology, independent of $t$. Since $dd^c s_i^*G_{n,t}$ converges to $dd^c s_i^*G_t$, there exists a positive real number $C_1>0$ such that for all $n$ and $t$,
\begin{align*}
    0<\int_{X_t}\left(dd^c\phi_{\cF}\wedge dd^cs_i^*G_{n,t}^-+dd^c \phi_{\cF} \wedge dd^c s_i^*G_t^+ \right)<C_1.
\end{align*}
Hence we have
\begin{align*}
    \left|\int_{X_t}\Phi_{\cF}d\eta_t-\int_{X_t}\Phi_{\cF}d\eta_{n,t}\right|\leq C_2\cdot\epsilon_n
\end{align*}
for some $C_2>0$ and $\epsilon_n\to 0$ as $n\to\infty$.

Now, we complete the proof. Take any $\epsilon>0$. By the above argument, there exists a large natural number $N_1$ such that $\epsilon_n<\epsilon/3C_2$ for any $n>N_1$. Also since $\eta_{n,0}$ weakly converges to $\eta_0$, there exists $N_2$ such that
\begin{align*}
    \left|\int_{X_{\bC((t))}^{\an}}g_{\cF}d\eta_{n,0}-\int_{X_{\bC((t))}^{\an}}g_{\cF}d\eta_0\right|<\frac{\epsilon}{3}.
\end{align*}
Fix any $n>\max(N_1,N_2)$. By Theorem \ref{thm: hybridK3-weakconvergence-invariantmeasure-modelfunctions} there exists a small $\delta>0$ so that for any $t\in\bar{\bD}_{\delta}$
\begin{align*}
    \left|\int_{X_t}\Phi_{\cF}d\eta_{n,t}-\int_{X_{\bC((t))}^{\an}}g_{\cF}d\eta_{n,0}\right|<\epsilon/3.
\end{align*}
Combining together, for any $t\in\bar{\bD}_{\delta}$ we have
\begin{align*}
    &\left|\int_{X_t}\Phi_{\cF}d\eta_t-\int_{X_{\bC((t))}^{\an}}g_{\cF}d\eta_0\right| \\
    &\leq\left| \int_{X_t}\Phi_{\cF}d\eta_t-\int_{X_t}\Phi_{\cF}d\eta_{n,t} \right|+\left|\int_{X_t}\Phi_{\cF}d\eta_{n,t}-\int_{X_{\bC((t))}^{\an}}g_{\cF}d\eta_{n,0}\right|+\left|\int_{X_{\bC((t))}^{\an}}g_{\cF}d\eta_{n,0}-\int_{X_{\bC((t))}^{\an}}g_{\cF}d\eta_0\right| \\
    &\leq \epsilon.
\end{align*}

\bibliographystyle{plain}
\bibliography{hybrid_dynamics_of_K3_surfaces}

\begin{thebibliography}{10}

\bibitem{Bedford:1976qd}
Eric Bedford and B.~A. Taylor.
\newblock The dirichlet problem for a complex monge-amp{\`e}re equation.
\newblock {\em Inventiones mathematicae}, 37(1):1--44, 1976.

\bibitem{Bianchi2018DegenerationOQ}
Fabrizio Bianchi and Y{\^u}suke Okuyama.
\newblock Degeneration of quadratic polynomial endomorphisms to a henon map.
\newblock {\em Indiana University Mathematics Journal}, 2018.

\bibitem{BoucksomEriksson2021}
Sebastien Boucksom and Dennis Eriksson.
\newblock Spaces of norms, determinant of cohomology and fekete points in non-archimedean geometry.
\newblock {\em Advances in Mathematics}, 378:107501, 02 2021.

\bibitem{Sebastien:2015uz}
S{\'e}bastien Boucksom, Charles Favre, and Mattias Jonsson.
\newblock Singular semipositive metrics in non-archimedean geometry.
\newblock {\em Journal of Algebraic Geometry}, 25(1):77--139, 2015.

\bibitem{BOUCKSOM:2015fr}
S{\'e}bastien Boucksom, Charles Favre, and Mattias Jonsson.
\newblock Solution to a non-archimedean monge-amp{\`e}re equation.
\newblock {\em Journal of the American Mathematical Society}, 28(3):617--667, 2015.

\bibitem{BHJ1}
S{\'e}bastien Boucksom, Tomoyuki Hisamoto, and Mattias Jonsson.
\newblock Uniform k-stability and asymptotics of energy functionals in kahler geometry.
\newblock {\em European Journal of Mathematics}, 21(9):2905--2944, 20179.

\bibitem{Sebastien:2017le}
S{\'e}bastien Boucksom and Mattias Jonsson.
\newblock Tropical and non-archimedean limits of degenerating families of volume forms.
\newblock {\em Journal de l'{\'E}cole polytechnique ---Math{\'e}matiques}, 4:87--139, 2017.

\bibitem{Cantat:2001}
Serge Cantat.
\newblock Dynamique des automorphismes des surfaces k3.
\newblock {\em Acta Mathematica}, 187(1):1--57, 2001.

\bibitem{chambertloir2012formes}
Antoine Chambert-Loir and Antoine Ducros.
\newblock Formes diff\'erentielles r\'eelles et courants sur les espaces de berkovich, 2012.

\bibitem{Demailly:1993if}
Jean-Pierre Demailly.
\newblock {\em Monge-Amp{\`e}re Operators, Lelong Numbers and Intersection Theory}, pages 115--193.
\newblock Springer US, Boston, MA, 1993.

\bibitem{DE-MARCO:2014tn}
Laura DeMarco and Xander Faber.
\newblock Degenerations of complex dynamical systems.
\newblock {\em Forum of Mathematics, Sigma}, 2:e6, 2014.

\bibitem{Favre:2020rz}
Charles Favre.
\newblock Degeneration of endomorphisms of the complex projective space in the hybrid space.
\newblock {\em Journal of the Institute of Mathematics of Jussieu}, 19(4):1141--1183, 2020.

\bibitem{Filip:2019ij}
Simion Filip.
\newblock Tropical dynamics of area-preserving maps.
\newblock {\em Journal of Modern Dynamics}, 14(0):179--226 EP --, 2019.

\bibitem{gromov2003entropy}
Mikha{\i}l Gromov.
\newblock On the entropy of holomorphic maps.
\newblock {\em Enseign. Math}, 49(3-4):217--235, 2003.

\bibitem{Huybrechts_2016}
Daniel Huybrechts.
\newblock {\em Lectures on K3 Surfaces}.
\newblock Cambridge Studies in Advanced Mathematics. Cambridge University Press, 2016.

\bibitem{irokawa2023hybrid}
Reimi Irokawa.
\newblock Hybrid dynamics of h\'enon mappings, 2023.

\bibitem{okuyama2024degenerating}
Y{\^u}suke Okuyama.
\newblock On a degenerating limit theorem of demarco--faber, 2024.

\end{thebibliography}
\end{document}